\documentclass{article}
\usepackage{arxiv}
\usepackage{graphicx} 
\usepackage{amsmath,amssymb,amsthm}
\newtheorem{thm}{Theorem}

\newtheorem{cor}{Corollary}

\title{On the Alon-Tarsi number of some line and total graphs}
\author{Prajnanaswaroopa S\\
sntrm4@rediffmail.com}
\date{}

\begin{document}

\maketitle

\section*{Abstract}
This work discusses the Alon-Tarsi number of line graphs and total graphs. In addition, we also discuss the Alon-Tarsi number of some Erdos-Faber-Lovasz (EFL) graphs.
\section*{Introduction}
The Combinatorial Nullstellensatz \cite{COMB} has become very popular in solving combinatorial as well as algebraic and number-theoretic problems of late. One of the main applications of the Combinatorial Nullstellensatz is in graph colorings, specifically list colorings. In this regard, we have the famous Alon-Tarsi number of a graph, introduced by Tarsi and Alon \cite{ALOT}, that gives an upper bound on the list chromatic number, or the choice number of a graph. Alon-Tarsi orientation of a graph is an orientation such that the number of odd and Eulerian subdigraphs differ. Alon-Tarsi number is the number defined by one more than the maximum out/ in degree in an Alon-Tarsi orientation of a graph. Every acyclic orientation of a graph is an example of an Alon-tarsi orientation. Essentially, translating to the language of polynomials, the Alon-Tarsi number of any polynomial is one more than the maximum exponent of any variable in the term having a minimum degree among the monomials of the graph polynomial. That is, we can define the Alon-Tarsi Number of any polynomial  $H=\sum_tc_t\mathbf{y_t}$ with $\mathbf{y_t}=y_{1}^{i_1}y_{2}^{i_2}\ldots y_n^{i_n}$ as $min_{t}(max_{i_k}(y_{1}^{i_1}y_{2}^{i_2}\ldots y_{n}^{i_n}))$. We use this polynomial version in this paper, which then implies that the Alon-Tarsi number of a graph $G$, denoted by $ATN(G)$ is the Alon-Tarsi number of its graph polynomial, which is the polynomial formed by multiplying all edge monomials of $G$, which, in turn, is formed taking each edge of $G$ in the form $x_i-x_j$ where $x_i$ and $x_j$ are the labeled vertices of the edge.\\

The famous list coloring conjecture (LCC) \cite{JEN} states that the choice number of any line graph equals its chromatic number. Or, in other words, any line graph of a graph is chromatic choosable.\\

The Erdos-Faber-Lovasz (EFL) graph of parameter $k$ is a class of graphs that are defined in the following way: We have $k$ cliques of order $k$, any two of which meet at most one vertex. Note that this can be seen as the graph-theoretical representation of a $k$-linear hypergraph of $k$. (This representation and other details are well discussed in \cite{JANZ}). The clique degree of any vertex in such a configuration is the number of cliques it is a part of. The popular Erdos-Faber-Lovasz (EFL)(\cite{KANG}, \cite{ERD}) conjecture in this regard is that such a graph, or for that matter, the edges of the linear hypergraph associated, can be colored in $k$ colors.\\

The total graph of a graph $G$, denoted by $T(G)$ \cite{BEH}, is a graph formed by subdividing all the edges of $G$ and connecting the vertices in the subdivided graph that correspond to incident edges of $G$ on the same vertex, as well as vertices which are adjacent in $G$. In this form, it can be seen as the $2-$ distance square of the bipartite graph $S(G)$, the subdivided graph of $G$, with one half square being the line graph $L(G)$ of $G$, and the other half square being $G$ itself. The Total coloring conjecture (TCC) \cite{VIZ}, \cite{BEH1} would mean that $\chi(T(G))\le\Delta(G)+2$. A weaker form of this, the weak TCC \cite{BAS} implies that $\chi(T(G))\le\Delta(G)+3$.\\

Other graph notations are pretty standard. $ch(G)$ denotes the choice number of $G$, and $\chi(G)$ denotes the chromatic number of $G$, $L(G)$, and $S(G)$ denote the line graph and subdivision graphs of $G$ respectively.

\section*{Theorems}
The following Theorem is taken from the paper \cite{PRA1}, Theorem 2.6. We state it verbatim for ease of understanding the following Theorem.
\begin{thm}
If $G$ is an $n$ order $1-$ factorizable regular graph with maximum degree $\Delta$ and $n=4k$ for some integer $k$, then $ATN$ of the line graph of $G$ is equal to $n-1$. Hence, the LCC holds for these graphs.
\end{thm}
\begin{proof}
The proof is similar to the Theorem 2.5 in \cite{PRA1}. Here, we partition the line graph into $\Delta$ independent sets, each having $\frac{n}{2}=2k$ vertices. Now, the arguments of the previous theorem hold, and we label the vertices of the line graph as $x_1,x_2,\ldots,x_{2k}$ for the first perfect matching; $x_{2k+1},x_{2k+2},\ldots,x_{n}$ and so on. The edges are now of the form $x_i-x_j$, where $i<j$. We also observe that an edge in any perfect matching is adjacent to exactly two edges in any other perfect matching (the adjacencies are determined by the end vertices). Thus, the line graphs of $ G $ are $\Delta$ -partite with any two partite sets inducing a regular bipartite graph of degree $ 2 $. Now, from Theorem 2.3 of \cite{PRA1}, the Alon-Tarsi monomial of the induced graph formed by any two partite sets in the line graph is of the form $c_i(x_{i_1}x_{i_2}\ldots x_{i_{4k}})^{\frac{2}{2}}=c_i(x_{i_1}x_{i_2}\ldots x_{i_{4k}})$ for some non-zero $c_i$ and indices $i_k$s. Since there are $(\Delta-1)$ such regular bipartite graphs (the induced graphs between any two perfect matchings) for each independent set of the line graph, we get that the final Alon-Tarsi monomial should be of the form $C(x_{i_1}x_{i_2}\ldots x_{i_{2k}})^{\Delta-1}(x_{i_{2k+1}}x_{i_{2k+2}}\ldots x_{4k})^{\Delta-2}\ldots(x_{i_{n-2k+1}}x_{i_{n-2k+2}}\ldots x_{i_{n}})$. This implies that the Alon-Tarsi number should be $(\Delta-1)+1=\Delta$. Hence, $G$ satisfies LCC.
\end{proof}
\begin{thm}
If $G$ is any graph, then $ATN(L(G))\le\Delta(G)+1$.
\end{thm}
\begin{proof}
Case:1\\
We consider the class $1$ graphs $G$ of order $n$. For the maximal case, we could consider the case when the graph is $1$-factorizable, which implies $G$ is regular of even order. Every class $1$ graph can be embedded in a $1$- factorizable graph of the same maximum degree. To see how, if $\Delta(G)$ was even, then we could take $\Delta(G)$ copies of $G$ and give an edge between the corresponding vertices of some two copies of $G$ until those corresponding vertices get their degree equal to the maximum degree. If $\Delta(G)$ was odd, we perform the prior step with $\Delta(G)+1$ copies instead.    Now, if $n=4k$ for some natural number $k$, we can use Theorem $1$ to get the conclusion. If $n\neq4k$, we can take two disjoint copies of $G$, forming a new graph $H$. Now, $H$ has $4k$ vertices, and again, the arguments of the Theorem $1$ would follow. This invariably implies that $G$ also has $ATN(L(G))=\Delta$. Thus, in any case, if $G$ is a class $1$ graph, $ATN(L(G))=\Delta$. Thus, these satisfy the LCC.\\
Case:2\\
In this case, we consider $G$ to be class $2$; that is, it cannot be edge colored in $\Delta(G)$ colors. Here, we could embed $G$ in a class $1$ graph having maximum degree $\Delta(G)+1$. To do this, we construct a new graph $H$ as follows. We form a new graph $G'$ by joining a new vertex and making some vertex adjacent to it. Now $G'$ is a class $1$ graph of maximum degree $\Delta(G)+1$. We use the procedure in the previous case to get $G'$ embedded in a $1$- factorizable graph $H$ of degree $\Delta(G)+1$. Thus, the number of edge colors required to color $H$ equals that of $G$. Now, we can apply the argument of the previous case, from which the Theorem is evident. 
\end{proof}
The previous two Theorems settle LCC for all graphs.
\begin{cor}
If $G$ is any graph, then let $T(G)$ be the total graph of $G$, which is formed by the disjoint union of $G$, line graph $L(G)$, and the edges of the subdivision graph $S(G)$. Then, $ATN(T(G))\le\Delta(G)+3$ 
\end{cor}
\begin{proof}
The proof is immediate, as the graph $T(G)$ can be seen as the $2$-distance (or square) graph of the subdivision graph, with one half square being $G$ and the other half square being $L(G)$. The interconnecting edges between the two half squares are the edges of the bipartite graph $S(G)$. As such, any minimal Alon-Tarsi monomial will have at most $2$ as an increment in the maximum exponent. Therefore, the Corollary follows.
\end{proof}
\begin{thm}
Any EFL graph $G$ of parameter $k$ such that either the induced graph formed by vertices having clique degree greater than or equal to $2$ has maximum degree $\le k-1$; or all the vertices have clique degree either $1$ or $2$, then $ATN(G)\le k$.
\end{thm}
\begin{proof}
We can see that EFL graphs can be seen as a union of the induced graph formed by the vertices having clique degrees greater than one (the contact vertices of the cliques), which we call $C$; several disjoint cliques of order less than or equal to $k-1$, which we call $D$, and the edges connecting the graph $C$ to $D$. Now, by arbitrary labeling, the graph polynomial of $D$, which is a product of Vandermonde determinants with several of the exponents being less than or equal to $k-1$, would always have a final monomial with the highest exponent less than or equal to $k-2$. Now, we consider the graph polynomial of $C$. 

When the clique degree is at most $2$, $G$ would be a union of the subgraph of a line graph with maximum degree $k-1$ and possibly some disjoint cliques, in which case, by the previous theorems, we get $ATN(G)\le k$.

Now, when $C$ has maximum degree $k-1$, the Alon-Tarsi monomial of $C$ would have maximum exponent $k-1$, in which case, the edges connecting $D$ and $C$ could be appropriately oriented to get us that the final Alon-Tarsi monomial would also be $k-1$. This is because, at every vertex of $C$, we have a star graph with other vertices of a clique in $D$, which has distinct monomials, with exponents varying from $0$ to $k-2$. We choose a monomial such that the sum of exponents of the star graph monomial and the Alon-Tarsi monomial of $C$ would be $\le k-1$. This gives us the required result.


\end{proof}

\end{document}